\documentclass[final,3p]{elsarticle}
\usepackage[latin1]{inputenc}
 \usepackage{graphics}
 \usepackage{graphicx}
 \usepackage{epsfig}
\usepackage{amssymb}
 \usepackage{amsthm}
 \usepackage{lineno}
 \usepackage{amsmath}
\usepackage{color}
   \numberwithin{equation}{section}
\usepackage{mathrsfs}

\journal{ } 

\newtheorem{thm}{Theorem}[section]

\newtheorem{lem}[thm]{Lemma}

\newtheorem{defn}[thm]{Definition}
\newtheorem{rem}[thm]{Remark}

\biboptions{sort&compress,square}
\allowdisplaybreaks
\begin{document}
\begin{frontmatter}
\author{Hongfeng Li}
\ead{lihf728@nenu.edu.cn}
\author{Yong Wang\corref{cor2}}
\ead{wangy581@nenu.edu.cn}
\cortext[cor2]{Corresponding author.}

\address{School of Mathematics and Statistics, Northeast Normal University,
Changchun, 130024, China}

\title{Spectral (0,4)-tensor functionals and the noncommutative residue}
\begin{abstract}
In this paper, we derive some spectral (0,4)-tensor functionals by four one-forms and the Dirac operator and the noncommutative residue on even-dimensional compact spin manifolds without boundary. Then, we extend these spectral (0,4)-tensor functionals to a general spectral triple.
\end{abstract}
\end{frontmatter}
\section{Introduction}
 Until now, many geometers have studied the noncommutative residue, which is of great importance to the study of noncommutative geometry. Connes showed that the noncommutative residue on a compact manifold $M$ coincided with the Dixmier's trace on pseudo-differential operators of order $-{\rm {dim}}M$ \cite{Co2}. Hence, the noncommutative residue can be used as integral of noncommutative geometry and become an important tool of noncommutative geometry. In \cite{Co1}, Connes used the noncommutative residue to derive a conformal 4-dimensional Polyakov action analogy.
Connes made a challenging observation that the noncommutative residue of the square of the inverse of the Dirac operator was proportional to the Einstein-Hilbert action, which is called the Kastler-Kalau-Walze type theorem. Kastler \cite{Ka} gave a brute-force proof of this theorem. In \cite{KW}, Kalau and Walze proved this theorem in the normal coordinates system simultaneously. Based on the theory of the noncommutative residue proposed by Wodzicki, Fedosov, et al. \cite{FGLS} constructed the noncommutative residue of the classical elemental algebra of the Boutte de Monvel calculus on compact manifolds of dimension $n>2$.

Using elliptic pseudo-differential operators and the noncommutative residue is a natural way to study the spectral Einstein functional and the operator-theoretic interpretation of the gravitational action on bounded manifolds. Concerning the Dirac operator and the signature operator, Wang carried out the computation of noncommutative residues and succeeded in proving the Kastler-Kalau-Walze type theorem for manifolds with boundary \cite{Wa1, Wa3, Wa4}. Figueroa, et al. \cite{FGV2} introduced a noncommutative integral based on the noncommutative residue. In \cite{DL}, Dabrowski, Sitarz and Zalecki defined the spectral Einstein functional for a general spectral triple and for the noncommutative torus, they computed the spectral Einstein functional.
In \cite{WWw}, Wang, et al. gave some new spectral functionals which are the extension of spectral functionals to the noncommutative realm with torsion, and related them to the noncommutative residue for manifolds with boundary about Dirac operators with torsion.
In \cite{DL2}, Dabrowski, et al. examined the metric and Einstein bilinear functionals of differential forms for the Hodge-Dirac operator $d+d^*$ on an oriented, closed, even-dimensional Riemannian manifold. Wu and Wang computed the spectral Einstein functional for the Witten deformation $d+d^*+\widehat{c}(V)$ on even-dimensional Riemannian manifolds without boundary \cite{Wu2}. In \cite{Li1}, Li and Wang computed the spectral Einstein functional associated with the nonminimal de Rham-Hodge operator on even-dimensional compact manifolds without boundary. They mainly computed the noncommutative residue $\mathrm{Wres}\Big(\widetilde{c}(u)\widetilde{c}(v)\widetilde{D}^{-2m+2}\Big)$ and $\mathrm{Wres}\Big(\widetilde{c}(u)\widetilde{D}\widetilde{c}(v)\widetilde{D}\widetilde{D}^{-2m}\Big)$. In \cite{B2}, the definition of spectral tensor type functionals were posed. We know that the spectral functionals in \cite{DL} are spectral (0,2)-tensor functionals. So it is interesting to find some spectral (0,4)-tensor functionals by the Dirac operator and the noncommutative residue.
{\bf The motivation of this paper} is to compute the noncommutative residue $\mathrm{Wres}\Big(c(u_1)c(u_2)c(u_3)c(u_4)D^{-2m+2}\Big)$ and $\mathrm{Wres}\Big(c(u_1)c(u_2)Dc(u_3)c(u_4)DD^{-2m}\Big)$ for the Dirac operator on even-dimensional Riemannian manifolds without boundary and get some interesting spectral (0,4)-tensor functionals. Our main theorem is as follows.
 \begin{thm}\label{1thm}
 	Let $M$ be an $n=2m$ dimensional ($n\geq 3$) Riemannian manifold, for the Dirac operator $D$, the following equalities hold
 \begin{align}
 	\mathscr{P}_{D}:=&\mathrm{Wres}\Big(c(u_1)c(u_2)c(u_3)c(u_4)D^{-2m+2}\Big)\nonumber\\
 =&-2^{m} \frac{2 \pi^{m}}{\Gamma\left(m\right)}\int_{M} \bigg\{\frac{m-1}{12}s\Big(g(u_1,u_2)g(u_3,u_4)-g(u_1,u_3)g(u_2,u_4)+g(u_1,u_4)g(u_2,u_3)\Big)\bigg\}d{\rm Vol}_M;\nonumber\\
 	\mathscr{Q}_{D}:=&\mathrm{Wres}\Big(c(u_1)c(u_2)Dc(u_3)c(u_4)DD^{-2m}\Big)\nonumber\\
=&2^{m} \frac{2 \pi^{m}}{\Gamma\left(m\right)}\int_{M}\frac{1}{24}\bigg\{2(m+1)sg(u_1,u_2)g(u_3,u_4)-2(m+2)s\Big(g(u_1,u_3)g(u_2,u_4)-g(u_1,u_4)g(u_2,u_3)\Big)\nonumber\\
&+12g(u_1,u_2){\rm Ric}(u_3,u_4)-3g(u_3,u_4){\rm Ric}(u_1,u_2)+g(u_1,u_3){\rm Ric}(u_2,u_4)-5g(u_2,u_4){\rm Ric}(u_1,u_3)\nonumber\\
&+17g(u_1,u_4){\rm Ric}(u_2,u_3)+2g(u_2,u_3){\rm Ric}(u_1,u_4)\bigg\}d{\rm Vol}_M,\nonumber
 \end{align}
 where  $g(u_i,u_j)=\sum_{a=1}^{n}u_{a}^i u_{a}^j $ and $c(u_i)=\sum_{a=1}^{n} u_a^i c(e_a)$ $(i,j=1,2,3,4)$.
 \end{thm}

The paper is organized in the following way. In Section \ref{section:2}, we introduce the definition of the Dirac operator and the symbols of the generalized laplacian for the Dirac operator. In Section \ref{section:3}, using the residue for a differential operator of Laplace type ${\rm Wres}(P):=\int_{S^*M}{\rm tr}(\sigma_{-n}^P)(x,\xi)$ and the composition formula of pseudo-differential operators, we obtain above spectral (0,4)-tensor functionals in Theorem \ref{1thm} on even-dimensional Riemannian manifolds without boundary.

\section{The Dirac operator and its symbols}
\label{section:2}
Firstly, we recall the definition of the Dirac operator. Let $M$ be an $n$-dimensional ($n\geq 3$) oriented compact spin Riemannian manifold with a Riemannian metric $g^{M}$.\\
\indent Let $\nabla^L$ be the Levi-Civita connection about $g^{M}$. In the
fixed orthonormal frame $\{e_1,\cdots,e_n\}$, the connection matrix $(\omega_{s,t})$ is defined by
\begin{equation}
\label{eq1}
\nabla^L(e_1,\cdots,e_n)= (e_1,\cdots,e_n)(\omega_{s,t}).\nonumber
\end{equation}
\indent Let $c(e_i)$ denotes the Clifford action, which satisfies
\begin{align}
\label{a4}
&c(e_i)c(e_j)+c(e_j)c(e_i)=-2g^{M}(e_i,e_j).
\end{align}

By \cite{Y}, the Dirac operator is given by
\begin{equation}
D=\sum^n_{i=1}c(e_i)\Big[e_i
-\frac{1}{4}\sum_{s,t}\omega_{s,t}(e_i)c(e_s)c(e_t)\Big].
\end{equation}

For a differential operator of Laplace type $P$, it has locally the form
\begin{equation}\label{p}
	P=-(g^{ij}\partial_i\partial_j+A^i\partial_i+B),
\end{equation}
where $\partial_{i}$  is a natural local frame on $TM,$ $(g^{ij})_{1\leq i,j\leq n}$ is the inverse matrix associated to the metric
matrix  $(g_{ij})_{1\leq i,j\leq n}$ on $M,$ $A^{i}$ and $B$ are smooth sections of $\textrm{End}(N)$ on $M$ (endomorphism).
If $P$ satisfies the form \eqref{p}, then there is a unique
connection $\nabla$ on $N$ and a unique endomorphism $\widetilde{E}$ such that
\begin{equation}
	P=-[g^{ij}(\nabla_{\partial_{i}}\nabla_{\partial_{j}}- \nabla_{\nabla^{L}_{\partial_{i}}\partial_{j}})+\widetilde{E}].
\end{equation}

Moreover
(with local frames of $T^{*}M$ and $N$), $\nabla_{\partial_{i}}=\partial_{i}+\omega_{i} $
and $\widetilde{E}$ are related to $g^{ij}$, $A^{i}$ and $B$ through
\begin{eqnarray}
	&&\omega_{i}=\frac{1}{2}g_{ij}\big(A^{i}+g^{kl}\Gamma_{ kl}^{j} \texttt{id}\big);\nonumber\\
	&&\widetilde{E}=B-g^{ij}\big(\partial_{i}(\omega_{j})+\omega_{i}\omega_{j}-\omega_{k}\Gamma_{ ij}^{k} \big),\nonumber
\end{eqnarray}
where $\Gamma_{ kl}^{j}$ is the  Christoffel coefficient of $\nabla^{L}$.

From \cite{DL}, we get
\begin{align} \label{dt2}
D^{2}=-g^{ab}(\nabla_{{\partial}_a}\nabla_{{\partial}_b}-\nabla_{{\nabla_{{\partial}_a}^{L}}\partial_b})+E,\nonumber
\end{align}
where $E=-\widetilde{E}, \nabla_{{\partial}_a}={\partial}_a-\widetilde{T}_a$.

Recall the definition of the Dirac operator $D^{2}$ in \cite{Ka}, \cite{KW} and \cite{Wa4}, we have
\begin{equation}
D^{2}=-\sum_{ij}g^{ij}\Big[\partial_{i}\partial_{j}+2\sigma_{i}\partial_{j}+(\partial_{i}\sigma_{j})+\sigma_{i}\sigma_{j}
    -\Gamma_{i,j}^{k}\partial_{k}-\Gamma_{i,j}^{k}\sigma_{k}\Big]+\frac{1}{4}s,
\end{equation}
where $\sigma_{i}:=-\frac{1}{4}\sum_{s,t}\omega_{s,t}(e_i)c(e_s)c(e_t)$.

Define $\omega_{s,t}
(e_i)=-\langle \nabla_{e_i}^{L}e_{s}, e_{t}\rangle $, we get
\begin{align} \label{ta1}
	\widetilde{T}_a&=-\frac{1}{4}\sum_{s, t=1}^{n}\langle \nabla_{{\partial}_a}^{L}e_{s}, e_{t}\rangle c(e_{s})c(e_{t});\nonumber\\
E&=\frac{1}{4}s.
\end{align}

In normal coordinates, $\widetilde{T}_a$ is expanded near $x=0$ by Taylor expansion. That is
\begin{align}
	\widetilde{T}_a={T}_a+{T}_{ab}x^b+O(x^2).\nonumber
\end{align}

By ${\partial}_{l}\langle \nabla_{{\partial}_a}^{L}e_{s}, e_{t}\rangle(x_0)=\frac{1}{2}{\rm R}_{lats}(x_0)$, we get
\begin{align} \label{at0}
	{T}_a=0,
\end{align}
and
\begin{align} \label{ta0}
	{T}_{ab}=-\frac{1}{8}\sum_{a,b,s,t=1}^{2m}R_{bats}(x_0) c(e_{s})c(e_{t}).
\end{align}

 \begin{lem}\label{lem1}\cite{DL}
	The leading symbols of the generalized laplacian $\Delta_{T,E}^{-m}$ are as follows:
	\begin{align}
		\sigma_{-2 m}(\Delta_{T,E}^{-m})=&\|\xi\|^{-2 m-2}\sum_{a,b,j,k=1}^{2m}\left(\delta_{a b}-\frac{m}{3} R_{a j b k} x^{j} x^{k}\right) \xi_{a} \xi_{b}+O\left(\mathbf{x}^{3}\right) ;\nonumber\\
		\sigma_{-2m-1}(\Delta_{T,E}^{-m})=& \frac{-2 m i}{3}\|\xi\|^{-2 m-2} \sum_{a,k=1}^{2m}\operatorname{Ric}_{a k} x^{k} \xi_{a}-2 m i\|\xi\|^{-2 m-2}\sum_{a,b=1}^{2m}\left(T_{a} \xi_{a}+T_{a b} x^{b} \xi_{a}\right)+O(\mathbf{x^2}) ;\nonumber\\
		\sigma_{-2m-2}(\Delta_{T,E}^{-m})=& \frac{m(m+1)}{3}\|\xi\|^{-2 m-4} \sum_{a,b=1}^{2m}\operatorname{Ric}_{a b} \xi_{a} \xi_{b}\nonumber \\
		&-2 m(m+1)\|\xi\|^{-2 m-4}\sum_{a,b=1}^{2m} T_{a} T_{b} \xi_{a} \xi_{b}+m\sum_{a,b=1}^{2m}\left(T_{a} T_{a}-T_{a a}\right)\|\xi\|^{-2 m-2}\nonumber \\
		&+2 m(m+1)\|\xi\|^{-2 m-4} \sum_{a,b=1}^{2m}T_{a b} \xi_{a} \xi_{b}-mE \|\xi\|^{-2m-2}+O(\mathbf{x}) ,\nonumber
	\end{align}
where $Ric$ denotes Ricci curvature.
\end{lem}

By \eqref{ta1}-\eqref{ta0} and Lemma \ref{lem1}, we get the following lemma.

\begin{lem}\label{lem2}\cite{DL}
	General dimensional symbols of the Dirac operator are given:
	\begin{align}
		\sigma_{-2 m}(D^{-2m})=&\|\xi\|^{-2 m-2}\sum_{a,b,j,k=1}^{2m} \left(\delta_{a b}-\frac{m}{3} R_{a j b k} x^{j} x^{k}\right) \xi_{a} \xi_{b}+O\left(\mathbf{x}^{3}\right); \nonumber\\
		\sigma_{-2m-1}(D^{-2m})=& -\frac{2 mi}{3}\|\xi\|^{-2 m-2}\sum_{a,b=1}^{2m}  \operatorname{Ric}_{a b} x^{b} \xi_{a}\nonumber\\
		&+\frac{ m i}{4}\|\xi\|^{-2 m-2} \sum_{a,b,s,t=1}^{2m} \operatorname{R}_{b a t s}(x_0) c(e_s) c(e_t) x^{b} \xi_{a}
		+O\left(\mathbf{x}^{2}\right);\nonumber\\
		\sigma_{-2m-2}(D^{-2m})=& \frac{m(m+1)}{3}\|\xi\|^{-2 m-4}\sum_{a,b=1}^{2m} \operatorname{Ric}_{a b} \xi_{a} \xi_{b} \nonumber\\
		&-\frac{ m (m+1)}{4}\|\xi\|^{-2 m-4} \sum_{a,b,s,t=1}^{2m}\operatorname{R}_{b a t s}(x_0) c(e_s) c(e_t) \xi_{a} \xi_{b}\nonumber\\
&-\frac{m}{4} \|\xi\|^{-2 m-2}s+O\left(\mathbf{x}\right),\nonumber
	\end{align}
where $i$ in $mi$ represents the imaginary unit, and $i$ in other subscript positions represents the $i$-th. 
\end{lem}

 \section{Spectral (0,4)-tensor functionals and the noncommutative residue}
 \label{section:3}
In this section, we want to obtain some spectral (0,4)-tensor functionals by the noncommutative residue for the Dirac operator.

For a pseudo-differential operator  $P$, acting on sections of a vector bundle over an $n$-dimensional compact Riemannian manifold $M$ ($n$ is even), the analogue of the volume element in noncommutative geometry is the operator  $D^{-n}=: d s^{n} $. And pertinent operators are realized as pseudo-differential operators on the spaces of sections. Extending previous definitions by Connes \cite{co5}, a noncommutative integral was introduced in \cite{FGV2} based on the noncommutative residue \cite{wo2}, combine (1.4) in \cite{co4} and \cite{Ka}, using the definition of the residue:
\begin{align}\label{wers}
	\int P d s^{n}:=\operatorname{Wres} P D^{-n}:=\int_{S^{*} M} \operatorname{tr}\left[\sigma_{-n}\left(P D^{-n}\right)\right](x, \xi),
\end{align}
where  $\sigma_{-n}\left(P D^{-n}\right) $ denotes the  $(-n)$th order piece of the complete symbols of  $P D^{-n} $,  $\operatorname{tr}$  as shorthand of trace.

Firstly, we review here technical tool of the computation, which are the integrals of polynomial functions over the unit spheres. By (32) in \cite{B1}, we define
\begin{align}
I_{S_n}^{\gamma_1\cdot\cdot\cdot\gamma_{2\bar{n}+2}}=\int_{|x|=1}d^nxx^{\gamma_1}\cdot\cdot\cdot x^{\gamma_{2\bar{n}+2}},
\end{align}
i.e. the monomial integrals over a unit sphere.
Then by Proposition A.2. in \cite{B1},  polynomial integrals over higher spheres in the $n$-dimesional case are given
\begin{align}
I_{S_n}^{\gamma_1\cdot\cdot\cdot\gamma_{2\bar{n}+2}}=\frac{1}{2\bar{n}+n}[\delta^{\gamma_1\gamma_2}I_{S_n}^{\gamma_3\cdot\cdot\cdot\gamma_{2\bar{n}+2}}+\cdot\cdot\cdot+\delta^{\gamma_1\gamma_{2\bar{n}+1}}I_{S_n}^{\gamma_2\cdot\cdot\cdot\gamma_{2\bar{n}+1}}],
\end{align}
where $S_n\equiv S^{n-1}$ in $\mathbb{R}^n$.

For $\bar{n}=0$, we have $I^0={\rm Vol}(S^{n-1})$=$\frac{2\pi^{\frac{n}{2}}}{\Gamma(\frac{n}{2})}$, we immediately get
\begin{align}
I_{S_n}^{\gamma_1\gamma_2}&=\frac{1}{n}{\rm Vol}(S^{n-1})\delta^{\gamma_1\gamma_2};\nonumber\\
I_{S_n}^{\gamma_1\gamma_2\gamma_3\gamma_4}&=\frac{1}{n(n+2)}{\rm Vol}(S^{n-1})[\delta^{\gamma_1\gamma_2}\delta^{\gamma_3\gamma_4}+\delta^{\gamma_1\gamma_3}\delta^{\gamma_2\gamma_4}+\delta^{\gamma_1\gamma_4}\delta^{\gamma_2\gamma_3}].
\end{align}

\begin{thm}\label{thm}
 	Let $M$ be an $n=2m$ dimensional ($n\geq 3$) Riemannian manifold, for the Dirac operator $D$, the following equalities hold
 \begin{align}
 	\mathscr{P}_{D}:=&\mathrm{Wres}\Big(c(u_1)c(u_2)c(u_3)c(u_4)D^{-2m+2}\Big)\nonumber\\
 =&-2^{m} \frac{2 \pi^{m}}{\Gamma\left(m\right)}\int_{M} \bigg\{\frac{m-1}{12}s\Big(g(u_1,u_2)g(u_3,u_4)-g(u_1,u_3)g(u_2,u_4)+g(u_1,u_4)g(u_2,u_3)\Big)\bigg\}d{\rm Vol}_M;\nonumber\\
 	\mathscr{Q}_{D}:=&\mathrm{Wres}\Big(c(u_1)c(u_2)Dc(u_3)c(u_4)DD^{-2m}\Big)\nonumber\\
=&2^{m} \frac{2 \pi^{m}}{\Gamma\left(m\right)}\int_{M}\frac{1}{24}\bigg\{2(m+1)sg(u_1,u_2)g(u_3,u_4)-2(m+2)s\Big(g(u_1,u_3)g(u_2,u_4)-g(u_1,u_4)g(u_2,u_3)\Big)\nonumber\\
&+12g(u_1,u_2){\rm Ric}(u_3,u_4)-3g(u_3,u_4){\rm Ric}(u_1,u_2)+g(u_1,u_3){\rm Ric}(u_2,u_4)-5g(u_2,u_4){\rm Ric}(u_1,u_3)\nonumber\\
&+17g(u_1,u_4){\rm Ric}(u_2,u_3)+2g(u_2,u_3){\rm Ric}(u_1,u_4)\bigg\}d{\rm Vol}_M,\nonumber
 \end{align}
where  $g(u_i,u_j)=\sum_{a=1}^{n}u_{a}^i u_{a}^j $ and $c(u_i)=\sum_{a=1}^{n} u_a^i c(e_a)$ $(i,j=1,2,3,4)$.
 \end{thm}

\begin{proof}

{\bf Part I)} $\mathscr{P}_{D}=\mathrm{Wres}\Big(c(u_1)c(u_2)c(u_3)c(u_4)D^{-2m+2}\Big)$.

By \eqref{wers}, we need to compute  $\int_{S^* M} \operatorname{tr}\left[\sigma_{-2 m}\Big(c(u_1)c(u_2)c(u_3)c(u_4)D^{-2m+2}\Big)\right](x, \xi) $. Based on the algorithm yielding the principal symbol of a product of pseudo-differential operators in terms of the principal symbols of the factors,  by Lemma \ref{lem2}, we have
\begin{align}
	\sigma_{-2m}(D^{-2 m+2})=& \frac{m(m-1)}{3}\|\xi\|^{-2 m-2}\sum_{a,b=1}^{2m} \operatorname{Ric}_{a b} \xi_{a} \xi_{b}\nonumber\\
		&-\frac{ m (m-1)}{4}\|\xi\|^{-2 m-2} \sum_{a,b,s,t=1}^{2m}\operatorname{R}_{b a t s} c(e_s) c(e_t) \xi_{a} \xi_{b}\nonumber\\
		&-\frac{m-1}{4} \|\xi\|^{-2 m}s+O\left(\mathbf{x}\right).
\end{align}
Obviously, the general dimensional symbols of the $c(u_1)c(u_2)c(u_3)c(u_4)D^{-2m+2}$ are given:
\begin{align}\label{PD}
	&\sigma_{-2 m}\Big(c(u_1)c(u_2)c(u_3)c(u_4)D^{-2m+2}\Big)\nonumber\\
 =& \frac{m(m-1)}{3}\|\xi\|^{-2 m-2}\sum_{a,b=1}^{2m} \operatorname{Ric}_{a b} \xi_{a} \xi_{b}c(u_1)c(u_2)c(u_3)c(u_4) \nonumber\\
		&-\frac{ m (m-1)}{4}\|\xi\|^{-2 m-2} \sum_{a,b,s,t=1}^{2m}\operatorname{R}_{b a t s}\xi_{a} \xi_{b} c(u_1)c(u_2)c(u_3)c(u_4)c(e_s) c(e_t) \nonumber\\
		&-\frac{m-1}{4} \|\xi\|^{-2 m}sc(u_1)c(u_2)c(u_3)c(u_4)+O\left(\mathbf{x}\right).
\end{align}
Next, we integrate each of the above items respectively.

\noindent {\bf (I-1)}
\begin{align}
	&\int_{\|\xi\|=1}\operatorname{tr} \biggl\{\frac{m(m-1)}{3}\|\xi\|^{-2 m-2}\sum_{a,b=1}^{2m} \operatorname{Ric}_{a b} \xi_{a} \xi_{b}c(u_1)c(u_2)c(u_3)c(u_4)\biggr\}(x_0)\sigma(\xi)\nonumber\\
	=&\frac{m(m-1)}{3}\delta_a^b \operatorname{Ric}_{a b}{\rm tr}[c(u_1)c(u_2)c(u_3)c(u_4)]{\rm Vol}(S^{n-1}).
\end{align}
By calculation, we get
\begin{align}\label{a1}
{\rm tr}[c(u_1)c(u_2)c(u_3)c(u_4)]=\Big(g(u_1,u_2)g(u_3,u_4)-g(u_1,u_3)g(u_2,u_4)+g(u_1,u_4)g(u_2,u_3)\Big){\rm tr}[id],
\end{align}
so
\begin{align}
	&\int_{\|\xi\|=1}\operatorname{tr} \biggl\{\frac{m(m-1)}{3}\|\xi\|^{-2 m-2}\sum_{a,b=1}^{2m} \operatorname{Ric}_{a b} \xi_{a} \xi_{b}c(u_1)c(u_2)c(u_3)c(u_4)\biggr\}(x_0)\sigma(\xi)\nonumber\\
	=&\frac{m(m-1)}{3}\delta_a^b \operatorname{Ric}_{a b}{\rm tr}[c(u_1)c(u_2)c(u_3)c(u_4)]{\rm Vol}(S^{n-1})\nonumber\\
=&\frac{m(m-1)}{3}\times\frac{1}{2m} \operatorname{Ric}_{a a}{\rm tr}[c(u_1)c(u_2)c(u_3)c(u_4)]{\rm Vol}(S^{n-1})\nonumber\\
=&\frac{m-1}{6}s\Big(g(u_1,u_2)g(u_3,u_4)-g(u_1,u_3)g(u_2,u_4)+g(u_1,u_4)g(u_2,u_3)\Big){\rm Vol}(S^{n-1}){\rm tr}[id].
\end{align}
\noindent {\bf (I-2)}
\begin{align}
	&\int_{\|\xi\|=1} {\rm tr}\biggl\{-\frac{ m (m-1)}{4}\|\xi\|^{-2 m-2} \sum_{a,b,s,t=1}^{2m}\operatorname{R}_{b a t s}\xi_{a} \xi_{b}c(u_1)c(u_2)c(u_3)c(u_4) c(e_s) c(e_t) \biggr\}(x_0)\sigma(\xi)\nonumber\\
	=&-\frac{ m-1}{8}\sum_{a,s,t=1}^{2m}\operatorname{R}_{a a t s}{\rm tr}\big[c(u_1)c(u_2)c(u_3)c(u_4) c(e_s) c(e_t)\big]{\rm Vol}(S^{n-1})  \nonumber\\
	=&0.
\end{align}
\noindent {\bf (I-3)}
\begin{align}
	&\int_{\|\xi\|=1} {\rm tr}\biggl\{-\frac{m-1}{4} \|\xi\|^{-2 m}sc(u_1)c(u_2)c(u_3)c(u_4)\biggr\}(x_0)\sigma(\xi)\nonumber\\
	=&-\frac{m-1}{4}s{\rm tr}\big[c(u_1)c(u_2)c(u_3)c(u_4)\big]{\rm Vol}(S^{n-1})\nonumber\\
=&-\frac{m-1}{4}s\Big(g(u_1,u_2)g(u_3,u_4)-g(u_1,u_3)g(u_2,u_4)+g(u_1,u_4)g(u_2,u_3)\Big){\rm Vol}(S^{n-1}){\rm tr}[id].
\end{align}
In summary, we have
\begin{align}\label{zpdt}
	&\int_{\|\xi\|=1} {\rm tr}\biggl\{\sigma_{-2 m}\Big(c(u_1)c(u_2)c(u_3)c(u_4)  D^{-2 m+2}\Big)\biggr\}(x_0)\sigma(\xi)\nonumber\\
	=&-\frac{m-1}{12}s\Big(g(u_1,u_2)g(u_3,u_4)-g(u_1,u_3)g(u_2,u_4)+g(u_1,u_4)g(u_2,u_3)\Big){\rm Vol}(S^{n-1}){\rm tr}[id].
\end{align}
Since ${\rm tr}[id]=2^{m}$ and ${\rm Vol}(S^{n-1})=\frac{2 \pi^{m}}{\Gamma\left(m\right)}$, we obtain
\begin{align}\label{z11}
	\mathscr{P}_{D}=&\mathrm{Wres}\Big(c(u_1)c(u_2)c(u_3)c(u_4)D^{-2m+2}\Big)\nonumber\\
	=&-2^{m} \frac{2 \pi^{m}}{\Gamma\left(m\right)}\int_{M} \bigg\{\frac{m-1}{12}s\Big(g(u_1,u_2)g(u_3,u_4)-g(u_1,u_3)g(u_2,u_4)+g(u_1,u_4)g(u_2,u_3)\Big)\bigg\}d{\rm Vol}_M.
\end{align}

{\bf Part II)} $\mathscr{Q}_{D}=\mathrm{Wres}\Big(c(u_1)c(u_2)Dc(u_3)c(u_4)DD^{-2m}\Big)$.

Let $c(u_1)c(u_2)D:=\mathcal{A},c(u_3)c(u_4)D:=\mathcal{B} .$
By \eqref{wers}, we need to compute  $\int_{S^* M} \operatorname{tr}\big[\sigma_{-2 m}(\mathcal{A} \mathcal{B} D
^{-2 m})\big](x, \xi) $. Based on the algorithm yielding the principal symbol of a product of pseudo-differential operators in terms of the principal symbols of the factors, we have
\begin{align}\label{ABD}
	&\sigma_{-2 m}\left(\mathcal{A} \mathcal{B} D^{-2 m}\right)  \nonumber\\
     =&\left\{\sum_{|\alpha|=0}^{\infty} \frac{(-i)^{|\alpha|}}{\alpha!} \partial_{\xi}^{\alpha}\big[\sigma(\mathcal{A} \mathcal{B})\big] \partial_{x}^{\alpha}\big[\sigma(D^{-2 m})\big]\right\}_{-2 m} \nonumber\\
	 =&\sigma_{0}(\mathcal{A} \mathcal{B}) \sigma_{-2 m}(D^{-2 m})+\sigma_{1}(\mathcal{A} \mathcal{B}) \sigma_{-2 m-1}(D^{-2 m})+\sigma_{2}(\mathcal{A} \mathcal{B}) \sigma_{-2 m-2}(D^{-2 m}) \nonumber\\
	& +(-i) \sum_{j=1}^{2m} \partial_{\xi_{j}}\big[\sigma_{2}(\mathcal{A} \mathcal{B})\big] \partial_{x_{j}}\big[\sigma_{-2 m-1}(D^{-2 m})\big]+(-i) \sum_{j=1}^{2m} \partial_{\xi_{j}}\big[\sigma_{1}(\mathcal{A} \mathcal{B})\big] \partial_{x_{j}}\big[\sigma_{-2 m}(D^{-2 m})\big] \nonumber\\
	& -\frac{1}{2} \sum_{j,l=1}^{2m} \partial_{\xi_{j}} \partial_{\xi_{l}}\big[\sigma_{2}(\mathcal{A} \mathcal{B})\big] \partial_{x_{j}} \partial_{x_{l}}\big[\sigma_{-2 m}(D^{-2 m})\big] .
\end{align}

 \begin{lem}
The symbols of  $\mathcal{A}$  and  $\mathcal{B}$  are given:
\begin{align}
&\sigma_{1}(\mathcal{A})=\sqrt{-1}c(u_1)c(u_2)c(\xi); \nonumber\\
	&\sigma_{1}(\mathcal{B})=\sqrt{-1}c(u_3)c(u_4)c(\xi);\nonumber\\
&\sigma_{0}(\mathcal{A})=-\frac{1}{4} \sum_{p,s,t=1}^{2m} w_{st}(e_p)c(u_1)c(u_2)c(e_p)c(e_s)c(e_t);\nonumber\\
	&\sigma_{0}(\mathcal{B})=-\frac{1}{4} \sum_{p,s,t=1}^{2m} w_{st}(e_p)c(u_3)c(u_4)c(e_p)c(e_s)c(e_t).\nonumber
\end{align}
 \end{lem}

Further, by the composition formula of pseudo-differential operators, we get the following lemma.

 \begin{lem}
	The symbols of $\mathcal{AB}$ are given:
	\begin{align}
		\sigma_{0}(\mathcal{AB})=&\sigma_{0}(\mathcal{A}) \sigma_{0}(\mathcal{B})+(-i) \partial_{\xi_{j}}\left[\sigma_{1}(\mathcal{A})\right] \partial_{x_{j}}\left[\sigma_{0}(\mathcal{B})\right]+(-i) \partial_{\xi_{j}}\left[\sigma_{0}(\mathcal{A})\right] \partial_{x_{j}}\left[\sigma_{1}(\mathcal{B})\right]\nonumber\\
		=&\frac{1}{16}\sum_{p,s,t,\hat{p},\hat{s},\hat{t}=1}^{2m} w_{st}(e_p) w_{\hat{s} \hat{t}}(e_{\hat{p}})c(u_1)c(u_2)c(e_p)c(e_s)c(e_t)c(u_3)c(u_4)c(e_{\hat{p}})c(e_{\hat{s}})c(e_{\hat{t}})\nonumber\\
&+\frac{1}{8}\sum_{j,\hat{p},\hat{s},\hat{t}=1}^{2m} {\operatorname{R}}_{j\hat{p}\hat{t}\hat{s}}c(u_1)c(u_2)c(dx_j)c(u_3)c(u_4)c(e_{\hat{p}})c(e_{\hat{s}})c(e_{\hat{t}})\nonumber\\
&-\frac{1}{4}\sum_{j,\hat{p},\hat{s},\hat{t},\gamma=1}^{2m}w_{\hat{s} \hat{t}}(e_{\hat{p}})\partial {x_j}(u_{\gamma}^3)c(u_1)c(u_2)c(dx_j)c(e_\gamma)c(u_4)c(e_{\hat{p}})c(e_{\hat{s}})c(e_{\hat{t}})\nonumber\\
&-\frac{1}{4}\sum_{j,\hat{p},\hat{s},\hat{t},{\hat{\gamma}}=1}^{2m}w_{\hat{s} \hat{t}}(e_{\hat{p}})\partial {x_j}(u_{\hat{\gamma}}^4)c(u_1)c(u_2)c(dx_j)c(u_3)c(e_{\hat{\gamma}})c(e_{\hat{p}})c(e_{\hat{s}})c(e_{\hat{t}});\nonumber\\
		\sigma_{1}(\mathcal{AB})=&\sigma_{1}(\mathcal{A}) \sigma_{0}(\mathcal{B})+\sigma_{0}(\mathcal{A}) \sigma_{1}(\mathcal{B})+(-i) \partial_{\xi_{j}}\left[\sigma_{1}(\mathcal{A})\right] \partial_{x_{j}}\left[\sigma_{1}(\mathcal{B})\right] \nonumber\\
		=&-\frac{i}{4}\sum_{p,s,t=1}^{2m} w_{st}(e_p) c(u_1)c(u_2)c(\xi)c(u_3)c(u_4)c(e_p)c(e_s)c(e_t)\nonumber\\
&-\frac{i}{4}\sum_{p,s,t=1}^{2m} w_{st}(e_p) c(u_1)c(u_2)c(e_p)c(e_s)c(e_t)c(u_3)c(u_4)c(\xi)\nonumber\\
		&+i\sum_{j,\gamma=1}^{2m}\partial x_j(u_{\gamma}^3)c(u_1)c(u_2)c(dx_j)c(e_\gamma)c(u_4)c(\xi)\nonumber\\
        &+i\sum_{j,{\widehat{\gamma}}=1}^{2m}\partial x_j(u_{\widehat{\gamma}}^4)c(u_1)c(u_2)c(dx_j)c(u_3)c(e_{\widehat{\gamma}})c(\xi);\nonumber\\
		\sigma_{2}(\mathcal{AB})=&\sigma_{1}(\mathcal{A})\sigma_{1}(\mathcal{B})=-c(u_1)c(u_2)c(\xi)c(u_3)c(u_4)c(\xi).\nonumber
	\end{align}
\end{lem}

\noindent {\bf (II-1)} For $\sigma_{0}(\mathcal{AB}) \sigma_{-2 m}(D^{-2 m})(x_{0})$:
\begin{align}\label{0-2m}
\sigma_{0}(\mathcal{AB}) \sigma_{-2 m}(D^{-2 m})(x_{0})
=&\frac{1}{8}\|\xi\|^{-2 m}\sum_{j,\hat{p},\hat{s},\hat{t}=1}^{2m}  {\operatorname{R}}_{j \hat{p} \hat{t} \hat{s}}c(u_1)c(u_2)c(dx_j)c(u_3)c(u_4)c(e_{\hat{p}})c(e_{\hat{s}})c(e_{\hat{t}}).
\end{align}
So
\begin{align}
	&\int_{\|\xi\|=1} \operatorname{tr}\big[\sigma_{0}(\mathcal{AB}) \sigma_{-2 m}(D^{-2 m})(x_{0})\big] \sigma(\xi)\nonumber\\
	=&\frac{1}{8}\sum_{j,\hat{p},\hat{s},\hat{t}=1}^{2m}  {\operatorname{R}}_{j \hat{p} \hat{t} \hat{s}}{\rm tr}\big[c(u_1)c(u_2)c(dx_j)c(u_3)c(u_4)c(e_{\hat{p}})c(e_{\hat{s}})c(e_{\hat{t}})\big]{\rm Vol}(S^{n-1}).
\end{align}
By calculation, we have
\begin{align}\label{t2}
	&{\rm tr}\big[c(u_1)c(u_2)c(dx_j)c(u_3)c(u_4)c(e_{\hat{p}})c(e_{\hat{s}})c(e_{\hat{t}})\big]\nonumber\\
=&\sum_{j,\hat{p},\hat{s},\hat{t},r,f,p,q=1}^{2m}u_ru_fu_pu_q\times{\rm tr}\big[c(e_r)c(e_f)c(e_j)c(e_p)c(e_q)c(e_{\hat{p}})c(e_{\hat{s}})c(e_{\hat{t}})\big]	\nonumber\\ =&\sum_{j,\hat{p},\hat{s},\hat{t},r,f,p,q=1}^{2m}u_ru_fu_pu_q\Big[\delta_r^f\delta_j^p\delta_q^{\hat{p}}\delta_{\hat{s}}^{\hat{t}}-\delta_r^f\delta_j^p\delta_q^{\hat{s}}\delta_{\hat{p}}^{\hat{t}}+\delta_r^f\delta_j^p\delta_q^{\hat{t}}\delta_{\hat{p}}^{\hat{s}}-\delta_r^f\delta_j^q\delta_p^{\hat{p}}\delta_{\hat{s}}^{\hat{t}}+\delta_r^f\delta_j^q\delta_p^{\hat{s}}\delta_{\hat{p}}^{\hat{t}}
-\delta_r^f\delta_j^q\delta_p^{\hat{t}}\delta_{\hat{p}}^{\hat{s}}+\delta_r^f\delta_j^{\hat{p}}\delta_p^q\delta_{\hat{s}}^{\hat{t}}\nonumber\\
&-\delta_r^f\delta_j^{\hat{p}}\delta_p^{\hat{s}}\delta_q^{\hat{t}}+\delta_r^f\delta_j^{\hat{p}}\delta_p^{\hat{t}}\delta_q^{\hat{s}}-\delta_r^f\delta_j^{\hat{s}}\delta_p^q\delta_{\hat{p}}^{\hat{t}}+\delta_r^f\delta_j^{\hat{s}}\delta_p^{\hat{p}}\delta_q^{\hat{t}}-\delta_r^f\delta_j^{\hat{s}}\delta_p^{\hat{t}}\delta_q^{\hat{p}}
+\delta_r^f\delta_j^{\hat{t}}\delta_p^q\delta_{\hat{p}}^{\hat{s}}-\delta_r^f\delta_j^{\hat{t}}\delta_p^{\hat{p}}\delta_q^{\hat{s}}+\delta_r^f\delta_j^{\hat{t}}\delta_p^{\hat{s}}\delta_q^{\hat{p}}-\delta_r^j\delta_f^p\delta_q^{\hat{p}}\delta_{\hat{s}}^{\hat{t}}\nonumber\\
&+\delta_r^j\delta_f^p\delta_q^{\hat{s}}\delta_{\hat{p}}^{\hat{t}}-\delta_r^j\delta_f^p\delta_q^{\hat{t}}\delta_{\hat{p}}^{\hat{s}}+\delta_r^j\delta_f^q\delta_p^{\hat{p}}\delta_{\hat{s}}^{\hat{t}}-\delta_r^j\delta_f^q\delta_p^{\hat{s}}\delta_{\hat{p}}^{\hat{t}}+\delta_r^j\delta_f^q\delta_p^{\hat{t}}\delta_{\hat{p}}^{\hat{s}}
-\delta_r^j\delta_f^{\hat{p}}\delta_p^q\delta_{\hat{s}}^{\hat{t}}+\delta_r^j\delta_f^{\hat{p}}\delta_p^{\hat{s}}\delta_q^{\hat{t}}-\delta_r^j\delta_f^{\hat{p}}\delta_p^{\hat{t}}\delta_q^{\hat{s}}+\delta_r^j\delta_f^{\hat{s}}\delta_p^q\delta_{\hat{p}}^{\hat{t}}\nonumber\\
&-\delta_r^j\delta_f^{\hat{s}}\delta_p^{\hat{p}}\delta_q^{\hat{t}}+\delta_r^j\delta_f^{\hat{s}}\delta_p^{\hat{t}}\delta_q^{\hat{p}}-\delta_r^j\delta_f^{\hat{t}}\delta_p^q\delta_{\hat{p}}^{\hat{s}}+\delta_r^j\delta_f^{\hat{t}}\delta_p^{\hat{p}}\delta_q^{\hat{s}}-\delta_r^j\delta_f^{\hat{t}}\delta_p^{\hat{s}}\delta_q^{\hat{p}}
+\delta_r^p\delta_f^j\delta_q^{\hat{p}}\delta_{\hat{s}}^{\hat{t}}-\delta_r^p\delta_f^{\hat{p}}\delta_j^q\delta_{\hat{s}}^{\hat{t}}+\delta_r^p\delta_f^j\delta_q^{\hat{t}}\delta_{\hat{p}}^{\hat{s}}-\delta_r^p\delta_f^q\delta_j^{\hat{p}}\delta_{\hat{s}}^{\hat{t}}\nonumber\\
&+\delta_r^p\delta_f^q\delta_j^{\hat{s}}\delta_{\hat{p}}^{\hat{t}}-\delta_r^p\delta_f^q\delta_j^{\hat{t}}\delta_{\hat{p}}^{\hat{s}}+\delta_r^p\delta_f^{\hat{p}}\delta_j^q\delta_{\hat{s}}^{\hat{t}}-\delta_r^p\delta_f^{\hat{p}}\delta_j^{\hat{s}}\delta_q^{\hat{t}}+\delta_r^p\delta_f^{\hat{p}}\delta_j^{\hat{t}}\delta_q^{\hat{s}}
-\delta_r^p\delta_f^{\hat{s}}\delta_j^q\delta_{\hat{p}}^{\hat{t}}+\delta_r^p\delta_f^{\hat{s}}\delta_j^{\hat{p}}\delta_q^{\hat{t}}-\delta_r^p\delta_f^{\hat{s}}\delta_j^{\hat{t}}\delta_q^{\hat{p}}+\delta_r^p\delta_f^{\hat{t}}\delta_j^q\delta_{\hat{p}}^{\hat{s}}\nonumber\\
&-\delta_r^p\delta_f^{\hat{t}}\delta_j^{\hat{p}}\delta_q^{\hat{s}}+\delta_r^p\delta_f^{\hat{t}}\delta_j^{\hat{s}}\delta_q^{\hat{p}}-\delta_r^q\delta_f^j\delta_p^{\hat{p}}\delta_{\hat{s}}^{\hat{t}}+\delta_r^q\delta_f^{\hat{p}}\delta_j^p\delta_{\hat{s}}^{\hat{t}}-\delta_r^q\delta_f^j\delta_p^{\hat{t}}\delta_{\hat{p}}^{\hat{s}}
+\delta_r^q\delta_f^p\delta_j^{\hat{p}}\delta_{\hat{s}}^{\hat{t}}-\delta_r^q\delta_f^p\delta_j^{\hat{s}}\delta_{\hat{p}}^{\hat{t}}+\delta_r^q\delta_f^p\delta_j^{\hat{t}}\delta_{\hat{p}}^{\hat{s}}-\delta_r^q\delta_f^{\hat{p}}\delta_j^p\delta_{\hat{s}}^{\hat{t}}\nonumber\\
&+\delta_r^q\delta_f^{\hat{p}}\delta_j^{\hat{s}}\delta_p^{\hat{t}}-\delta_r^q\delta_f^{\hat{p}}\delta_j^{\hat{t}}\delta_p^{\hat{s}}+\delta_r^q\delta_f^{\hat{s}}\delta_j^p\delta_{\hat{p}}^{\hat{t}}-\delta_r^q\delta_f^{\hat{s}}\delta_j^{\hat{p}}\delta_p^{\hat{t}}+\delta_r^q\delta_f^{\hat{s}}\delta_j^{\hat{t}}\delta_p^{\hat{p}}
-\delta_r^q\delta_f^{\hat{t}}\delta_j^p\delta_{\hat{p}}^{\hat{s}}+\delta_r^q\delta_f^{\hat{t}}\delta_j^{\hat{p}}\delta_p^{\hat{s}}-\delta_r^q\delta_f^{\hat{t}}\delta_j^{\hat{s}}\delta_p^{\hat{p}}+\delta_r^{\hat{p}}\delta_f^j\delta_p^q\delta_{\hat{s}}^{\hat{t}}\nonumber\\
&-\delta_r^{\hat{p}}\delta_f^q\delta_j^p\delta_{\hat{s}}^{\hat{t}}+\delta_r^{\hat{p}}\delta_f^j\delta_p^{\hat{t}}\delta_q^{\hat{s}}-\delta_r^{\hat{p}}\delta_f^p\delta_j^q\delta_{\hat{s}}^{\hat{t}}+\delta_r^{\hat{p}}\delta_f^p\delta_j^{\hat{s}}\delta_q^{\hat{t}}-\delta_r^{\hat{p}}\delta_f^p\delta_j^{\hat{t}}\delta_q^{\hat{s}}
+\delta_r^{\hat{p}}\delta_f^q\delta_j^p\delta_{\hat{s}}^{\hat{t}}-\delta_r^{\hat{p}}\delta_f^q\delta_j^{\hat{s}}\delta_p^{\hat{t}}+\delta_r^{\hat{p}}\delta_f^q\delta_j^{\hat{t}}\delta_p^{\hat{s}}-\delta_r^{\hat{p}}\delta_f^{\hat{s}}\delta_j^p\delta_q^{\hat{t}}\nonumber\\
&+\delta_r^{\hat{p}}\delta_f^{\hat{s}}\delta_j^q\delta_p^{\hat{t}}-\delta_r^{\hat{p}}\delta_f^{\hat{s}}\delta_j^{\hat{t}}\delta_p^q+\delta_r^{\hat{p}}\delta_f^{\hat{t}}\delta_j^p\delta_q^{\hat{s}}-\delta_r^{\hat{p}}\delta_f^{\hat{t}}\delta_j^{\hat{q}}\delta_p^{\hat{s}}+\delta_r^{\hat{p}}\delta_f^{\hat{t}}\delta_j^{\hat{s}}\delta_p^q
-\delta_r^{\hat{s}}\delta_f^j\delta_p^q\delta_{\hat{p}}^{\hat{t}}+\delta_r^{\hat{s}}\delta_f^q\delta_j^p\delta_{\hat{p}}^{\hat{t}}-\delta_r^{\hat{s}}\delta_f^j\delta_p^{\hat{t}}\delta_q^{\hat{p}}+\delta_r^{\hat{s}}\delta_f^p\delta_j^q\delta_{\hat{p}}^{\hat{t}}\nonumber\\
&-\delta_r^{\hat{s}}\delta_f^p\delta_j^{\hat{p}}\delta_q^{\hat{t}}+\delta_r^{\hat{s}}\delta_f^p\delta_j^{\hat{t}}\delta_q^{\hat{p}}-\delta_r^{\hat{s}}\delta_f^q\delta_j^p\delta_{\hat{p}}^{\hat{t}}+\delta_r^{\hat{s}}\delta_f^q\delta_j^{\hat{p}}\delta_p^{\hat{t}}-\delta_r^{\hat{s}}\delta_f^q\delta_j^{\hat{t}}\delta_p^{\hat{p}}
+\delta_r^{\hat{s}}\delta_f^{\hat{p}}\delta_j^p\delta_q^{\hat{t}}-\delta_r^{\hat{s}}\delta_f^{\hat{p}}\delta_j^q\delta_p^{\hat{t}}+\delta_r^{\hat{s}}\delta_f^{\hat{p}}\delta_j^{\hat{t}}\delta_p^q-\delta_r^{\hat{s}}\delta_f^{\hat{t}}\delta_j^p\delta_q^{\hat{p}}\nonumber\\
&+\delta_r^{\hat{s}}\delta_f^{\hat{t}}\delta_j^q\delta_p^{\hat{p}}-\delta_r^{\hat{s}}\delta_f^{\hat{t}}\delta_j^{\hat{p}}\delta_p^q+\delta_r^{\hat{t}}\delta_f^j\delta_p^q\delta_{\hat{p}}^{\hat{s}}-\delta_r^{\hat{t}}\delta_f^q\delta_j^p\delta_{\hat{p}}^{\hat{s}}+\delta_r^{\hat{t}}\delta_f^j\delta_p^{\hat{s}}\delta_q^{\hat{p}}
-\delta_r^{\hat{t}}\delta_f^p\delta_j^q\delta_{\hat{p}}^{\hat{s}}+\delta_r^{\hat{t}}\delta_f^p\delta_j^{\hat{p}}\delta_q^{\hat{s}}-\delta_r^{\hat{t}}\delta_f^p\delta_j^{\hat{s}}\delta_q^{\hat{p}}+\delta_r^{\hat{t}}\delta_f^q\delta_j^p\delta_{\hat{p}}^{\hat{s}}\nonumber\\
&-\delta_r^{\hat{t}}\delta_f^q\delta_j^{\hat{p}}\delta_p^{\hat{s}}+\delta_r^{\hat{t}}\delta_f^q\delta_j^{\hat{s}}\delta_p^{\hat{p}}-\delta_r^{\hat{t}}\delta_f^{\hat{p}}\delta_j^p\delta_q^{\hat{s}}+\delta_r^{\hat{t}}\delta_f^{\hat{p}}\delta_j^q\delta_p^{\hat{s}}-\delta_r^{\hat{t}}\delta_f^{\hat{p}}\delta_j^{\hat{s}}\delta_p^q
+\delta_r^{\hat{t}}\delta_f^{\hat{s}}\delta_j^p\delta_q^{\hat{p}}-\delta_r^{\hat{t}}\delta_f^{\hat{s}}\delta_j^q\delta_p^{\hat{p}}+\delta_r^{\hat{t}}\delta_f^{\hat{s}}\delta_j^{\hat{p}}\delta_p^q
\Big]{\rm tr}[id],
\end{align}
then
\begin{align}
	&\sum_{j,\hat{p},\hat{s},\hat{t}=1}^{2m}  {\operatorname{R}}_{j \hat{p} \hat{t} \hat{s}}{\rm tr}\big[c(u_1)c(u_2)c(dx_j)c(u_3)c(u_4)c(e_{\hat{p}})c(e_{\hat{s}})c(e_{\hat{t}})\big]\nonumber\\
=&\Big(2sg(u_1,u_2)g(u_3,u_4)+4g(u_1,u_2){\rm Ric}(u_3,u_4)-g(u_3,u_4){\rm Ric}(u_1,u_2)+3g(u_1,u_3){\rm Ric}(u_2,u_4)\nonumber\\
&+g(u_2,u_4){\rm Ric}(u_1,u_3)+3g(u_1,u_4){\rm Ric}(u_2,u_3)-2g(u_2,u_3){\rm Ric}(u_1,u_4)\Big){\rm tr}[id].
\end{align}
Therefore, we get
\begin{align}
	&\int_{\|\xi\|=1} \operatorname{tr}\big[\sigma_{0}(\mathcal{AB}) \sigma_{-2 m}(D^{-2 m})(x_{0})\big] \sigma(\xi)\nonumber\\
	=&\frac{1}{8}\Big(2sg(u_1,u_2)g(u_3,u_4)+4g(u_1,u_2){\rm Ric}(u_3,u_4)-g(u_3,u_4){\rm Ric}(u_1,u_2)+3g(u_1,u_3){\rm Ric}(u_2,u_4)\nonumber\\
&+g(u_2,u_4){\rm Ric}(u_1,u_3)+3g(u_1,u_4){\rm Ric}(u_2,u_3)-2g(u_2,u_3){\rm Ric}(u_1,u_4)\Big){\rm Vol}(S^{n-1}){\rm tr}[id].
\end{align}

\noindent{\bf (II-2)} For $\sigma_{1}(\mathcal{AB}) \sigma_{-2 m-1}(D^{-2 m})(x_{0})$:\\

\noindent Obviously, we have
\begin{align}
	&\sigma_{1}(\mathcal{AB}) \sigma_{-2 m-1}(D^{-2 m})(x_{0})=0,\nonumber
\end{align}
so
\begin{align}
	&\int_{\|\xi\|=1} \operatorname{tr}\big[\sigma_{1}(\mathcal{AB}) \sigma_{-2 m-1}(D^{-2 m})(x_{0})\big] \sigma(\xi)=0.\nonumber
\end{align}

\noindent{\bf (II-3)} For $\sigma_{2}(\mathcal{AB}) \sigma_{-2 m-2}(D^{-2 m})(x_{0})$:
\begin{align}\label{2-2m-2}
&\sigma_{2}(\mathcal{AB}) \sigma_{-2 m-2}(D^{-2 m})(x_{0})\nonumber\\
=&-\frac{m(m+1)}{3}\|\xi\|^{-2 m-4}\sum_{a,b=1}^{2m} \operatorname{Ric}_{a b} \xi_{a} \xi_{b}c(u_1)c(u_2)c(\xi)c(u_3)c(u_4)c(\xi) \nonumber\\
&+\frac{ m (m+1)}{4}\|\xi\|^{-2 m-4} \sum_{a,b,s,t=1}^{2m}\operatorname{R}_{b a t s} \xi_{a} \xi_{b}c(u_1)c(u_2)c(\xi)c(u_3)c(u_4)c(\xi)c(e_s) c(e_t)\nonumber\\
&+\frac{m}{4}\|\xi\|^{-2 m-2}sc(u_1)c(u_2)c(\xi)c(u_3)c(u_4)c(\xi).
\end{align}

\noindent{\bf (II-3-$\mathbb{A}$)}Due to
\begin{align}\label{t7}
	{\rm tr}\big[c(u_1)c(u_2)c(\xi)c(u_3)c(u_4)c(\xi)\big]=\sum_{f,g=1}^{2m}\xi_f\xi_g{\rm tr}\big[c(u_1)c(u_2)c(e_f)c(u_3)c(u_4)c(e_g)\big],
\end{align}
then
\begin{align}\label{a5}
	&\int_{\|\xi\|=1} \operatorname{tr}\biggl\{-\frac{m(m+1)}{3}\|\xi\|^{-2 m-4}\sum_{a,b=1}^{2m} \operatorname{Ric}_{a b} \xi_{a} \xi_{b} c(u_1)c(u_2)c(\xi)c(u_3)c(u_4)c(\xi)\biggr\}(x_0)\sigma(\xi)\nonumber\\
=&-\frac{m(m+1)}{3}\times\frac{1}{2m(2m+2)}\sum_{a,b,f,g=1}^{2m}\operatorname{Ric}_{a b}\big(\delta_a^b\delta_f^g+\delta_a^f\delta_b^g+\delta_a^g\delta_b^f\big)\nonumber\\
&\times{\rm tr}\big[c(u_1)c(u_2)c(e_f)c(u_3)c(u_4)c(e_g)\big]{\rm Vol}(S^{n-1})\nonumber\\
=&-\frac{1}{12}\sum_{l,a,f=1}^{2m}R(e_l,e_a,e_l,e_a){\rm tr}\big[c(u_1)c(u_2)c(e_f)c(u_3)c(u_4)c(e_f)\big]{\rm Vol}(S^{n-1})\nonumber\\
&-\frac{1}{12}\sum_{l,a,b=1}^{2m}R(e_l,e_a,e_l,e_b){\rm tr}\big[c(u_1)c(u_2)c(e_a)c(u_3)c(u_4)c(e_b)\big]{\rm Vol}(S^{n-1})\nonumber\\
&-\frac{1}{12}\sum_{l,a,b=1}^{2m}R(e_l,e_a,e_l,e_b){\rm tr}\big[c(u_1)c(u_2)c(e_b)c(u_3)c(u_4)c(e_a)\big]{\rm Vol}(S^{n-1}).
\end{align}
By \eqref{a1} and straightforward computations yield
\begin{align}\label{a2}
&{\rm tr}\big[c(u_1)c(u_2)c(e_f)c(u_3)c(u_4)c(e_f)\big]\nonumber\\
=&u_ru_{f_1}u_pu_q \times {\rm tr}\big[c(e_r)c(e_{f_1})c(e_f)c(e_p)c(e_q)c(e_f)\big]\nonumber\\
=&u_ru_{f_1}u_pu_q \times \Big(-4g(e_r,e_{f_1})g(e_p,e_q){\rm tr}[id]+(2m+4){\rm tr}\big[c(e_r)c(e_{f_1})c(e_p)c(e_q)\big]\Big)\nonumber\\
=&-4g(u_1,u_2)g(u_3,u_4){\rm tr}[id]+(2m+4){\rm tr}\big[c(u_1)c(u_2)c(u_3)c(u_4)\big]\nonumber\\
=&\Big(2mg(u_1,u_2)g(u_3,u_4)-(2m+4)g(u_1,u_3)g(u_2,u_4)+(2m+4)g(u_1,u_4)g(u_2,u_3)\Big){\rm tr}[id],
\end{align}
and
\begin{align}\label{a3}
&{\rm tr}\big[c(u_1)c(u_2)c(e_a)c(u_3)c(u_4)c(e_b)\big]\nonumber\\
=&u_ru_fu_pu_q \times {\rm tr}\big[c(e_r)c(e_f)c(e_a)c(e_p)c(e_q)c(e_b)\big]\nonumber\\
=&u_ru_fu_pu_q \times \Big(-\delta_r^f\delta_a^p\delta_q^b+\delta_r^f\delta_a^q\delta_p^b-\delta_r^f\delta_a^b\delta_p^q+\delta_r^a\delta_f^p\delta_q^b-\delta_r^a\delta_f^q\delta_p^b
+\delta_r^a\delta_f^b\delta_p^q-\delta_r^p\delta_f^a\delta_q^b+\delta_r^p\delta_f^q\delta_a^b\nonumber\\
&-\delta_r^p\delta_f^b\delta_a^q+\delta_r^q\delta_f^a\delta_p^b
-\delta_r^q\delta_f^p\delta_a^b+\delta_r^q\delta_f^b\delta_a^p-\delta_r^b\delta_f^a\delta_p^q+\delta_r^b\delta_f^p\delta_a^q-\delta_r^b\delta_f^q\delta_a^p\Big){\rm tr}[id]\nonumber\\
=&\Big(-\delta_a^bu_ru_fu_pu_q+2\big(u_fu_p-u_fu_q-u_ru_p+u_ru_q\big)u_au_b\Big){\rm tr}[id].
\end{align}
Similarly, we have
\begin{align}\label{a4}
&{\rm tr}\big[c(u_1)c(u_2)c(e_b)c(u_3)c(u_4)c(e_a)\big]\nonumber\\
=&\Big(-\delta_a^bu_ru_fu_pu_q+2\big(u_fu_p-u_fu_q-u_ru_p+u_ru_q\big)u_au_b\Big){\rm tr}[id].
\end{align}
Substituting \eqref{a2}-\eqref{a4} into \eqref{a5}, we get
\begin{align}
	&\int_{\|\xi\|=1} \operatorname{tr}\biggl\{-\frac{m(m+1)}{3}\|\xi\|^{-2 m-4}\sum_{a,b=1}^{2m} \operatorname{Ric}_{a b} \xi_{a} \xi_{b} c(u_1)c(u_2)c(\xi)c(u_3)c(u_4)c(\xi)\biggr\}(x_0)\sigma(\xi)\nonumber\\
=&\bigg[-\frac{m-1}{6}sg(u_1,u_2)g(u_3,u_4)+\frac{m+2}{6}s\Big(g(u_1,u_3)g(u_2,u_4)-g(u_1,u_4)g(u_2,u_3)\Big)\nonumber\\
&-\frac{1}{3}\Big(g(u_2,u_3){\rm Ric}(u_1,u_4)-g(u_2,u_4){\rm Ric}(u_1,u_3)-g(u_1,u_3){\rm Ric}(u_2,u_4)\nonumber\\
&+g(u_1,u_4){\rm Ric}(u_2,u_3)\Big)\bigg]{\rm Vol}(S^{n-1}){\rm tr}[id].
\end{align}

\noindent{\bf (II-3-$\mathbb{B}$)}

\begin{align}
	&\int_{\|\xi\|=1}\operatorname{tr} \biggl\{\frac{ m (m+1)}{4}\|\xi\|^{-2 m-4} \sum_{a,b,s,t=1}^{2m}\operatorname{R}_{b a t s} \xi_{a} \xi_{b}c(u_1)c(u_2)c(\xi)c(u_3)c(u_4)c(\xi)c(e_s) c(e_t) \biggr\}(x_0)\sigma(\xi)\nonumber\\
=&\frac{ m (m+1)}{4}\times\frac{1}{2m(2m+2)}\sum_{a,b,f,g,s,t=1}^{2m}\operatorname{R}_{b a t s} \big(\delta_a^b\delta_f^g+\delta_a^f\delta_b^g+\delta_a^g\delta_b^f\big) \nonumber\\
&\times{\rm tr}\big[c(u_1)c(u_2)c(e_f)c(u_3)c(u_4)c(e_g)c(e_s) c(e_t)\big]{\rm Vol}(S^{n-1})\nonumber\\
=&\frac{1}{16}\sum_{a,b,s,t=1}^{2m}\operatorname{R}_{b a t s} {\rm tr}\big[c(u_1)c(u_2)c(e_a)c(u_3)c(u_4)c(e_b)c(e_s) c(e_t)\big]{\rm Vol}(S^{n-1}) \nonumber\\
&+\frac{1}{16}\sum_{a,b,s,t=1}^{2m}\operatorname{R}_{b a t s} {\rm tr}\big[c(u_1)c(u_2)c(e_b)c(u_3)c(u_4)c(e_a)c(e_s) c(e_t)\big]{\rm Vol}(S^{n-1}) \nonumber\\
	=&0.\nonumber
\end{align}

\noindent{\bf (II-3-$\mathbb{C}$)}
By \eqref{a2}, we get
\begin{align}
	&\int_{\|\xi\|=1}\operatorname{tr} \biggl\{\frac{m}{4}\|\xi\|^{-2 m-2}sc(u_1)c(u_2)c(\xi)c(u_3)c(u_4)c(\xi)\biggr\}(x_0)\sigma(\xi)\nonumber\\
	=&\int_{\|\xi\|=1}\operatorname{tr} \biggl\{\frac{m}{4}\|\xi\|^{-2 m-2}s\xi_{f} \xi_{g}c(u_1)c(u_2)c(e_f)c(u_3)c(u_4)c(e_g)\biggr\}(x_0)\sigma(\xi)\nonumber\\
	=&\frac{1}{8}s{\rm tr}\big[c(u_1)c(u_2)c(e_f)c(u_3)c(u_4)c(e_f)\big]{\rm Vol}(S^{n-1}) \nonumber\\
	=&\bigg[\frac{m}{4}sg(u_1,u_2)g(u_3,u_4)-\frac{m+2}{4}s\Big(g(u_1,u_3)g(u_2,u_4)-g(u_1,u_4)g(u_2,u_3)\Big)\bigg]{\rm tr}[id].
\end{align}
As a result, we obtain
\begin{align}
	&\int_{\|\xi\|=1} \operatorname{tr}\big[\sigma_{2}(\mathcal{AB}) \sigma_{-2 m-2}(D^{-2 m})(x_{0})\big] \sigma(\xi)\nonumber\\
	=&\bigg\{\frac{m+2}{12}s\Big(g(u_1,u_2)g(u_3,u_4)-g(u_1,u_3)g(u_2,u_4)-g(u_1,u_4)g(u_2,u_3)\Big)-\frac{1}{3}\Big(g(u_2,u_3){\rm Ric}(u_1,u_4)\nonumber\\
&-g(u_2,u_4){\rm Ric}(u_1,u_3)-g(u_1,u_3){\rm Ric}(u_2,u_4)+g(u_1,u_4){\rm Ric}(u_2,u_3)\Big)\bigg\}{\rm Vol}(S^{n-1}){\rm tr}[id].
\end{align}

\noindent{\bf (II-4)} For $-i\sum_{j=1}^{2m} \partial_{\xi_{j}}\big[\sigma_{2}(\mathcal{A} \mathcal{B})\big] \partial_{x_{j}}\big[\sigma_{-2 m-1}(D^{-2 m})\big](x_{0})$:

\begin{align}\label{2-2m-1}
	&-i\sum_{j=1}^{2m} \partial_{\xi_{j}}[\sigma_{2}(\mathcal{A} \mathcal{B})] \partial_{x_{j}}[\sigma_{-2 m-1}(D^{-2 m})](x_{0})\nonumber\\
	=&\;\;\frac{2 m }{3}\|\xi\|^{-2 m-2}\sum_{a,b,j=1}^{2m}  \operatorname{Ric}_{a b}  \xi_{a}\delta^{b}_{j} \Big(c(u_1)c(u_2)c(dx_j)c(u_3)c(u_4)c(\xi)\nonumber\\
&+c(u_1)c(u_2)c(\xi)c(u_3)c(u_4)c(dx_j)\Big)\nonumber\\
	&-\frac{ m }{4}\|\xi\|^{-2 m-2} \sum_{a,b,j,s,t=1}^{2m} \operatorname{R}_{b a t s}\xi_{a}\delta^{b}_{j}\Big(c(u_1)c(u_2)c(dx_j)c(u_3)c(u_4)c(\xi)\nonumber\\
&+c(u_1)c(u_2)c(\xi)c(u_3)c(u_4)c(dx_j) \Big) c(e_s) c(e_t).
\end{align}

\noindent{\bf (II-4-$\mathbb{A}$)}
By \eqref{a4} and straightforward computations yield
\begin{align}
	&\int_{\|\xi\|=1} {\rm tr}\biggl\{\frac{2 m }{3}\|\xi\|^{-2 m-2}\sum_{a,b,j=1}^{2m}  \operatorname{Ric}_{a b}  \xi_{a}\delta^{b}_{j}\Big(c(u_1)c(u_2)c(dx_j)c(u_3)c(u_4)c(\xi)\nonumber\\
&+c(u_1)c(u_2)c(\xi)c(u_3)c(u_4)c(dx_j)\Big)\biggr\}(x_0)\sigma(\xi)\nonumber\\
	=&\int_{\|\xi\|=1}{\rm tr} \biggl\{\frac{2 m }{3}\|\xi\|^{-2 m-2}\sum_{a,b,g=1}^{2m}  \operatorname{Ric}_{a b}  \xi_{a}\xi_{g}\Big(c(u_1)c(u_2)c(e_b)c(u_3)c(u_4)c(e_g)\nonumber\\
&+c(u_1)c(u_2)c(e_g)c(u_3)c(u_4)c(e_b)\Big)\biggr\}(x_0)\sigma(\xi)\nonumber\\
	=&\frac{1}{3}\sum_{a,b=1}^{2m}  \operatorname{Ric}_{a b}  {\rm tr}\big[c(u_1)c(u_2)c(e_b)c(u_3)c(u_4)c(e_a)+c(u_1)c(u_2)c(e_a)c(u_3)c(u_4)c(e_b)\big]{\rm Vol}(S^{n-1})\nonumber\\
	=&\frac{2}{3}\sum_{a,b=1}^{2m}  \operatorname{Ric}_{a b}  {\rm tr}\big[c(u_1)c(u_2)c(e_b)c(u_3)c(u_4)c(e_a)\big]{\rm Vol}(S^{n-1})\nonumber\\
	=&\bigg[-\frac{2}{3}sg(u_1,u_2)g(u_3,u_4)+\frac{4}{3}\Big(g(u_2,u_3){\rm Ric}(u_1,u_4)-g(u_2,u_4){\rm Ric}(u_1,u_3)\nonumber\\
&-g(u_1,u_3){\rm Ric}(u_2,u_4)+g(u_1,u_4){\rm Ric}(u_2,u_3)\Big)\bigg]{\rm Vol}(S^{n-1}){\rm tr}[id].
\end{align}

\noindent{\bf (II-4-$\mathbb{B}$)}

\begin{align}
	&\int_{\|\xi\|=1}\operatorname{tr} \biggl\{-\frac{ m }{4}\|\xi\|^{-2 m-2} \sum_{a,b,j,s,t=1}^{2m} \operatorname{R}_{b a t s}\xi_{a}\delta^{b}_{j}\Big(c(u_1)c(u_2)c(dx_j)c(u_3)c(u_4)c(\xi)\nonumber\\
&+c(u_1)c(u_2)c(\xi)c(u_3)c(u_4)c(dx_j)\Big) c(e_s) c(e_t)\biggr\}(x_0)\sigma(\xi)\nonumber\\
	=&\int_{\|\xi\|=1}\operatorname{tr} \biggl\{-\frac{ m }{4}\|\xi\|^{-2 m-2} \sum_{a,b,s,t,g=1}^{2m} \operatorname{R}_{b a t s}\xi_{a}\xi_{g}\Big(c(u_1)c(u_2)c(e_b)c(u_3)c(u_4)c(e_g)\nonumber\\
&+c(u_1)c(u_2)c(e_g)c(u_3)c(u_4)c(e_b) \Big)c(e_s) c(e_t)\biggr\}(x_0)\sigma(\xi)\nonumber\\
=&-\frac{1 }{8}\sum_{a,b,s,t=1}^{2m} \operatorname{R}_{b a t s}\operatorname{tr}\big[c(u_1)c(u_2)c(e_b)c(u_3)c(u_4)c(e_a) c(e_s) c(e_t)\big]{\rm Vol}(S^{n-1})\nonumber\\
&-\frac{1 }{8}\sum_{a,b,s,t=1}^{2m} \operatorname{R}_{ab t s}\operatorname{tr}\big[c(u_1)c(u_2)c(e_b)c(u_3)c(u_4)c(e_a) c(e_s) c(e_t)\big]{\rm Vol}(S^{n-1})\nonumber\\
	=&0.\nonumber
\end{align}
In summary, we get
\begin{align}
	&\int_{\|\xi\|=1} \operatorname{tr}\bigg[-i \sum_{j=1}^{2m} \partial_{\xi_{j}}\big[\sigma_{2}(\mathcal{A} \mathcal{B})\big] \partial_{x_{j}}\big[\sigma_{-2 m-1}(D^{-2 m})\big]\bigg] (x_0)\sigma(\xi)\nonumber\\
	=&\bigg[-\frac{2}{3}sg(u_1,u_2)g(u_3,u_4)+\frac{4}{3}\Big(g(u_2,u_3){\rm Ric}(u_1,u_4)-g(u_2,u_4){\rm Ric}(u_1,u_3)\nonumber\\
&-g(u_1,u_3){\rm Ric}(u_2,u_4)+g(u_1,u_4){\rm Ric}(u_2,u_3)\Big)\bigg]{\rm Vol}(S^{n-1}){\rm tr}[id].
\end{align}

\noindent{\bf (II-5)} For $-i \sum_{j=1}^{2m} \partial_{\xi_{j}}\big[\sigma_{1}(\mathcal{A} \mathcal{B})\big] \partial_{x_{j}}\big[\sigma_{-2 m}(D^{-2 m})\big](x_{0})$:
\begin{align}\label{1-2m}
	\partial_{x_{j}}\big[\sigma_{-2 m}(D^{-2 m})\big](x_0)
	=&\partial_{x_{j}}\bigg[\|\xi\|^{-2 m-2}\sum_{a,b,l,k=1}^{2m} \left(\delta_{a b}-\frac{m}{3} R_{a l b k} x^{l} x^{k}\right) \xi_{a} \xi_{b}\bigg](x_0)\nonumber\\
	=&0,\nonumber
\end{align}
so
\begin{align}
	&\int_{\|\xi\|=1} \operatorname{tr}\bigg[-i\sum_{j=1}^{2m} \partial_{\xi_{j}}\big[\sigma_{1}(\mathcal{A} \mathcal{B})\big] \partial_{x_{j}}\big[\sigma_{-2 m}(D^{-2 m})\big]\bigg] (x_0)\sigma(\xi)=0.\nonumber
\end{align}

\noindent{\bf (II-6)} For $-\frac{1}{2} \sum_{j ,l=1}^{2m} \partial_{\xi_{j}} \partial_{\xi_{l}}\big[\sigma_{2}(\mathcal{A} \mathcal{B})\big] \partial_{x_{j}} \partial_{x_{l}}\big[\sigma_{-2 m}(D^{-2 m})\big](x_{0})$:

\begin{align}
\sum_{j, l=1} ^{2m}\partial_{\xi_{j}} \partial_{\xi_{l}}\big[\sigma_{2}(\mathcal{A} \mathcal{B})\big](x_{0})
=&\sum_{j, l=1} ^{2m}\partial_{\xi_{j}} \partial_{\xi_{l}}\big[-c(u_1)c(u_2)c(\xi)c(u_3)c(u_4)c(\xi)\big]\nonumber\\
=&-\sum_{j, l=1} ^{2m}\big[c(u_1)c(u_2)c(dx_l)c(u_3)c(u_4)c(dx_j)\nonumber\\
&+c(u_1)c(u_2)c(dx_j)c(u_3)c(u_4)c(dx_l)\big],
\end{align}
and
\begin{align}
\sum_{j, l=1} ^{2m}\partial_{x_{j}} \partial_{x_{l}}\big[\sigma_{-2 m}(D^{-2 m})\big](x_{0})
&=\sum_{j, l=1} ^{2m}\partial_{x_{j}} \partial_{x_{l}}\bigg[\|\xi\|^{-2 m-2}\sum_{a,b,\hat{j},k=1}^{2m} \left(\delta_{a b}-\frac{m}{3} R_{a \hat{j} b k} x^{\hat{j}} x^{k}\right) \xi_{a} \xi_{b}\bigg]\nonumber\\
&=-\frac{m}{3}\|\xi\|^{-2 m-2}\sum_{j, l,a,b,\hat{j},k=1} ^{2m}\Big(R_{a\hat{j}bk}\delta_l^{\hat{j}}\delta_j^k+R_{a\hat{j}bk}\delta_l^k\delta_j^{\hat{j}}\Big)\xi_a\xi_b.
\end{align}
So
\begin{align}\label{1-2m}
	&-\frac{1}{2} \sum_{j, l=1} ^{2m}\partial_{\xi_{j}} \partial_{\xi_{l}}\big[\sigma_{2}(\mathcal{A} \mathcal{B})\big] \partial_{x_{j}} \partial_{x_{l}}\big[\sigma_{-2 m}(D^{-2 m})\big](x_{0})\nonumber\\
	=&-\frac{m}{6}\|\xi\|^{-2 m-2}\sum_{j, l,a,b,\hat{j},k=1} ^{2m}\Big(R_{a\hat{j}bk}\delta_l^{\hat{j}}\delta_j^k+R_{a\hat{j}bk}\delta_l^k\delta_j^{\hat{j}}\Big)\xi_a\xi_bc(u_1)c(u_2)c(dx_l)c(u_3)c(u_4)c(dx_j)\nonumber\\
&-\frac{m}{6}\|\xi\|^{-2 m-2}\sum_{j, l,a,b,\hat{j},k=1} ^{2m}\Big(R_{a\hat{j}bk}\delta_l^{\hat{j}}\delta_j^k+R_{a\hat{j}bk}\delta_l^k\delta_j^{\hat{j}}\Big)\xi_a\xi_bc(u_1)c(u_2)c(dx_j)c(u_3)c(u_4)c(dx_l),
\end{align}
then
\begin{align}
	&\int_{\|\xi\|=1} {\rm tr}\biggl\{-\frac{1}{2} \sum_{j, l=1} ^{2m}\partial_{\xi_{j}} \partial_{\xi_{l}}\big[\sigma_{2}(\mathcal{A} \mathcal{B})\big] \partial_{x_{j}} \partial_{x_{l}}\big[\sigma_{-2 m}(D^{-2 m})\big]\biggr\}(x_0)\sigma(\xi)\nonumber\\
	=&\int_{\|\xi\|=1} \biggl\{-\frac{m}{6}\|\xi\|^{-2 m-2}\sum_{a,b,j,l=1}^{2m}\Big({\rm R}_{albj}+{\rm R}_{ajbl} \Big)\xi_{a}\xi_{b}{\rm tr}\big[c(u_1)c(u_2)c(dx_l)c(u_3)c(u_4)c(dx_j)\big]\biggr\}(x_0)\sigma(\xi)\nonumber\\
&+\int_{\|\xi\|=1} \biggl\{-\frac{m}{6}\|\xi\|^{-2 m-2}\sum_{a,b,j,l=1}^{2m}\Big({\rm R}_{albj}+{\rm R}_{ajbl} \Big)\xi_{a}\xi_{b}{\rm tr}\big[c(u_1)c(u_2)c(dx_j)c(u_3)c(u_4)c(dx_l)\big]\biggr\}(x_0)\sigma(\xi)\nonumber\\
	=&-\frac{1}{6}\sum_{a,j,l=1}^{2m}{\rm R}_{alaj}{\rm tr}\big[c(u_1)c(u_2)c(dx_l)c(u_3)c(u_4)c(dx_j)+c(u_1)c(u_2)c(dx_j)c(u_3)c(u_4)c(dx_l)\big]{\rm Vol}(S^{n-1})\nonumber\\
	=&\bigg[\frac{1}{3}sg(u_1,u_2)g(u_3,u_4)-\frac{2}{3}\Big(g(u_2,u_3){\rm Ric}(u_1,u_4)-g(u_2,u_4){\rm Ric}(u_1,u_3)-g(u_1,u_3){\rm Ric}(u_2,u_4)\nonumber\\
&+g(u_1,u_4){\rm Ric}(u_2,u_3)\Big)\bigg]{\rm Vol}(S^{n-1}){\rm tr}[id].
\end{align}
Thus, by summing {\bf (II-1)} to {\bf (II-6)}, we get
\begin{align}\label{zabdt}
	&\int_{\|\xi\|=1} {\rm tr}\biggl\{	\sigma_{-2 m}(\mathcal{A} \mathcal{B} D^{-2 m})\biggr\}(x_0)\sigma(\xi)\nonumber\\
	=&\frac{1}{24}\bigg\{2(m+1)sg(u_1,u_2)g(u_3,u_4)-2(m+2)s\Big(g(u_1,u_3)g(u_2,u_4)-g(u_1,u_4)g(u_2,u_3)\Big)\nonumber\\
&+12g(u_1,u_2){\rm Ric}(u_3,u_4)-3g(u_3,u_4){\rm Ric}(u_1,u_2)+g(u_1,u_3){\rm Ric}(u_2,u_4)-5g(u_2,u_4){\rm Ric}(u_1,u_3)\nonumber\\
&+17g(u_1,u_4){\rm Ric}(u_2,u_3)+2g(u_2,u_3){\rm Ric}(u_1,u_4)\bigg\} {\rm Vol}(S^{n-1}){\rm tr}[id].
\end{align}
Further, we obtain
\begin{align}\label{z22}
	\mathscr{Q}_{D}=&\mathrm{Wres}\Big(c(u_1)c(u_2)Dc(u_3)c(u_4)DD^{-2m}\Big)\nonumber\\
	=&2^{m} \frac{2 \pi^{m}}{\Gamma\left(m\right)}\int_{M}\frac{1}{24}\bigg\{2(m+1)sg(u_1,u_2)g(u_3,u_4)-2(m+2)s\Big(g(u_1,u_3)g(u_2,u_4)-g(u_1,u_4)g(u_2,u_3)\Big)\nonumber\\
&+12g(u_1,u_2){\rm Ric}(u_3,u_4)-3g(u_3,u_4){\rm Ric}(u_1,u_2)+g(u_1,u_3){\rm Ric}(u_2,u_4)-5g(u_2,u_4){\rm Ric}(u_1,u_3)\nonumber\\
&+17g(u_1,u_4){\rm Ric}(u_2,u_3)+2g(u_2,u_3){\rm Ric}(u_1,u_4)\bigg\}d{\rm Vol}_M.
\end{align}
Hence, by \eqref{z11} and \eqref{z22}, Theorem \ref{thm} holds.
\end{proof}

Let $(\mathcal{A},\mathcal{H},D)$ be an $n$-summable unital non-self-adjoint spectral triple, where $\mathcal{A}$ is a noncommutative algebra with involution, acting in the Hilbert space $\mathcal{H}$ while $D$ is a Dirac operator, which maybe non-self-adjoint operator and such that
$[D,a]$ is bounded $\forall a \in \mathcal{A}$. We assume that there exists a generalised algebra of
pseudo-differential operators, which contains $\mathcal{A},D,$ $D^l$ for $l\in \mathbb{Z}$ and there exists a tracial state $\mathcal{W}$  on it, called a noncommutative residue. Moreover, we assume that the noncommutative residue identically vanishes on $TD^{-k}$ for any $k>2m$ and a zero-order operator $T$.

\begin{defn}
For a regular spectral triple, we define spectral (0,4)-tensor functionals by
\begin{align}
&\mathrm{Wres}\Big(c(u_1)c(u_2)c(u_3)c(u_4)D^{-2m+2}\Big);\nonumber\\
&\mathrm{Wres}\Big(c(u_1)c(u_2)Dc(u_3)c(u_4)DD^{-2m}\Big).\nonumber
\end{align}
\end{defn}

\begin{rem}
1)It is interesting to compute above spectral (0,4)-tensor functionals for the noncommutative spectral triple, for examples the noncommutative torus and almost commutative spectral triple.\\
2)It is interesting to compute similar functionals
\begin{align}
&\mathrm{Wres}\Big(c(u_1)Dc(u_2)c(u_3)c(u_4)D^{-2m+1}\Big);\nonumber\\
&\mathrm{Wres}\Big(c(u_1)c(u_2)c(u_3)Dc(u_4)D^{-2m+1}\Big);\nonumber\\
&\mathrm{Wres}\Big(c(u_1)c(u_2)Dc(u_3)Dc(u_4)D^{-2m}\Big).\nonumber
\end{align}
\end{rem}

\section*{Declarations}
\textbf{Conflict of interest} The authors declare no conflicts of interest.

\section*{Acknowledgements}
This work was supported by the National Natural Science Foundation of China (No.11771070).

\end{document}